\documentclass[11pt,a4paper]{amsart}
\usepackage{amsmath,amsthm,amsopn,amsfonts,graphicx,amssymb,dsfont,color}
\usepackage{amsaddr}
\usepackage{subcaption}
\usepackage{mathrsfs}
\usepackage{color}
\usepackage[pdftex,
colorlinks=true,
linkcolor=blue,
citecolor=blue
]{hyperref}
\usepackage{enumerate}
\usepackage{a4,a4wide}

\usepackage[pagewise]{lineno}

\providecommand{\abs}[1]{\ensuremath{{\left| #1 \right|}}}%
\providecommand{\norme}[1]{\ensuremath{{\left\lVert #1 \right\rVert}}}%
\def\R{\mathbb R}
\def\N{\mathbb N}

\newcommand{\ubar}{\overline{u}}

\newcommand{\diffcst}{\sigma ^{2}}

\usepackage{letltxmacro}
\let\oldsqrt\sqrt
\def\sqrt{\mathpalette\DHLhksqrt}
\def\DHLhksqrt#1#2{%
\setbox0=\hbox{$#1\oldsqrt{#2\,}$}\dimen0=\ht0
\advance\dimen0-0.2\ht0
\setbox2=\hbox{\vrule height\ht0 depth -\dimen0}%
{\box0\lower0.4pt\box2}}

\newtheorem{thm}{\textbf{Theorem}}[section]
\newtheorem{lem}[thm]{\textbf{Lemma}}
\newtheorem{prop}[thm]{\textbf{Proposition}}

\newtheorem{cor}[thm]{\textbf{Corollary}}
\newtheorem{defi}[thm]{\textbf{Definition}}

\theoremstyle{remark}
\newtheorem{rem}[thm]{\textbf{Remark}}

\numberwithin{equation}{section}

\title{\bf Density dependent replicator-mutator models in directed evolution}
\author{Matthieu Alfaro {\tiny and} Mario Veruete}
\address{\textsc{imag}, \emph{Universit\'e de Montpellier}, \textsc{cnrs}, Montpellier, France.}
\email{matthieu.alfaro@umontpellier.fr}

\email{mario.veruete@umontpellier.fr}

\date{\today}

\begin{document}
\maketitle

\begin{abstract} 
We analyze a replicator-mutator model arising in the context of directed evolution \cite{Zad-17}, where the selection term is modulated over time by the mean-f{}itness. We combine a Cumulant Generating Function approach~\cite{Gil-17} and a spatio-temporal rescaling related to the Avron-Herbst formula~\cite{Alf-Car-14} to give of a complete picture of the Cauchy problem. Besides its well-posedness, we provide an implicit/explicit expression of the solution, and analyze its large time behaviour. As a by product, we also solve a replicator-mutator model where the  mutation coef{}f{}icient is socially determined, in the sense that it is modulated by the mean-f{}itness. The latter model reveals concentration or anti dif{}fusion/dif{}fusion  phenomena.
\end{abstract}

\tableofcontents
\newpage

\section{Introduction}\label{s:introduction}




In this paper, we analyze a mathematical model of a directed evolution process which is density-dependent. The model in consideration is derived in \cite{Zad-17} (see below for details) and is given by the following nonlocal equation
\begin{equation}\label{eq:principale}
\frac{\partial u}{\partial t}(t,x)=\diffcst \frac{\partial ^2u}{\partial x^2}(t,x) +u(t,x)\, \frac{x-\overline x(t)}{\overline x (t)}, \quad t>0,\; x\in \R,
\end{equation}
where the nonlocal term is given by
\begin{equation}\label{eq:premier-moment}
\overline x(t):=\int _\R x\, u(t,x)\,dx,
\end{equation}
and will also be denoted $\overline u(t)$, to remind its dependence on the solution itself. We will prove the well-posedness of the Cauchy problem associated with~\eqref{eq:principale}, obtain the solution through an implicit/explicit expression, and analyze its long time behaviour. 

As a by product of our analysis of~\eqref{eq:principale}, we will collect results on the well-posedness and the long time behaviour of the solution to the Cauchy problem associated with equation
\begin{equation}\label{eq:v-equation}
\frac{\partial v}{\partial t}(t,x)=\diffcst \overline{x}(t)\frac{\partial ^2v}{\partial x^2}(t,x) +v(t,x)(x-\overline{x}(t)), \quad t>0,\; x\in \R, \end{equation} where $\overline x (t):=\int_\R xv(t,x)\,dx$ will also be denoted $\overline v(t)$.

\medskip

Throughout this work, we assume that the initial data $u_0$ and $v_0$, are non-negative and have unitary mass \begin{equation}
\int _\R u_0(x)\,dx=\int _\R v_0(x)\,dx=1,
\end{equation} so that, \emph{formally}, $\int _\R u(t,x)\,dx=\int _\R v(t,x)\,dx=1$ for later times. Indeed, if we
formally integrate~\eqref{eq:principale} over $x\in \R$, we see that the total mass $M(t):=\int _\R u(t,x)\,dx$ solves the Cauchy problem
\begin{equation}\label{eq:ode-masse}
M'(t)=1-M(t), \quad M(0)=1,
\end{equation} so that, by the Cauchy-Lipschitz theorem,  $M(t)\equiv1$ for all $t>0$. Hence, $u(t,\cdot)$ is a probability density on $\R$ and $\overline x(t)=\overline u(t)$ is its mean. As far as~\eqref{eq:v-equation} is concerned, we reach
\begin{equation}\label{eq:v-ode-masse}
M'(t)=\overline v(t)(1-M(t)), \quad M(0)=1,
\end{equation} so that, by Gronwall's lemma, $M(t)=1$ as long as ``\,$\overline v(t)$ is meaningful\,''. We refer to~\cite{Alf-Car-14} for situations where, in some sense, $\overline v(t)$ blows up in f{}inite time, thus leading to extinction of the solution, which contradicts the formal conservation of the mass.

\medskip

Let us now comment on the emergence of equations~\eqref{eq:principale} and~\eqref{eq:v-equation}, termed as \emph{replicator-mutator} models,  in evolutionary biology and their relevance in  biotechnology.

Replicator-mutator models aim at describing Darwinian evolutionary processes,  whose fundamental elements are mutations and selection. Originally introduced by  Taylor and Jonker~\cite{TAYLOR1978145}, the replicator dynamics was developed for symmetric games in order to describe the evolution (determined by the payof{}f matrix) of the frequencies of strategies in a population, see~\cite{bookHofbauer} for a complete derivation.  Nevertheless, such an approach neglects the ef{}fect of mutations. As an attempt to f{}ill this gap, replicator-mutator models have been developed, starting with the work of Kimura~\cite{Kimura1965} and followed by models considering dif{}ferent types of mutations, both in the local~\cite{Biktashev2014}, \cite{Alf-Car-14,Alf-Car-17}, \cite{Alf-Ver-18} and nonlocal~\cite{Kimura1965}, \cite{Fle-79}, \cite{Bur-86, Bur-88-discret, Bur-88}, \cite{Gil-17,Gil-18} cases. 

It is important to mention the diversity of research areas where this type of model is applied:  population genetics ~\cite{Hadeler81}, game theory~\cite{Bomze95}, auto-catalytic reaction networks~\cite{Stadler1992} and language evolution~\cite{Nowak2001}.  As pointed out by Schuster and Sigmund~\cite{SCHUSTER1983533}, in the ordinary differential equation case, several evolutionary models in dif{}ferent biological f{}ields lead independently to the same class of replicator dynamics, showing some sort of universal structure; this idea is also developed in~\cite{PAGE200293}, where authors show how  apparently very dif{}ferent formulations of evolutionary dynamics are part of a single unif{}ied framework given by the replicator-mutator equation.

When mutations are modeled by the local diffusion operator,  and under the constrain of mass, the replicator-mutator equation  typically takes the form
\begin{equation}\label{modele}
    \frac{\partial u}{\partial t} = \underbrace{\diffcst \frac{\partial^2 u}{\partial x^2}}_{\text{mutations}} + \underbrace{u\left( \mathcal{W}(x) - \int_\R \mathcal{W}(y) u(t,y)\, dy\right)}_{\text{selection}}.
\end{equation}
In the context of evolutionary biology, $u(t,x)$ represents the density of population, at time $t$,  per unit of phenotypic trait  on a one-dimensional phenotypic trait space. The function $\mathcal{W}(x)$ stands for the f{}itness of the phenotype $x$ and models the individual reproductive success. Hence the nonlocal term $\overline{u}(t)=\overline{\mathcal{W}}(t)=\int_\R \mathcal{W}(y) u(t,y)\, dy$ is the mean f{}itness at time $t$, and can be seen as a Lagrange multiplier for the mass to be conserved, thus yielding an equation on the frequencies.

Due to their \emph{nonlocal} structure, the analysis of replicator-mutator equations often requires new approaches compared with local-\textsc{pde} (e.g. comparison principle does not hold), both from the analytic and numerical viewpoint. In the framework of~\eqref{modele}, we mention the works~\cite{Alf-Car-17}, \cite{Alf-Ver-18} for the cases of quadratic or conf{}ining  f{}itness functions, but now stick to the case  where the f{}itness linearly depends on the phenotypic trait, say $\mathcal W(x)=x$.

In this setting, the authors of~\cite{Alf-Car-14} proved that the solution of the replicator mutator Cauchy problem~\eqref{modele} with linear f{}itness can be written explicitly in terms of the solution of the Heat equation.  On the other hand, when a probability density (or kernel) $J$ models mutations, the equation is recast 
\begin{equation*}
\frac{\partial u}{\partial t}  =J* u-u +\,  u \left(x-\int _\R y\,u(t,y)\,dy\right),
\end{equation*}
for which a {\it Cumulant Generating Functions} (CGF) approach has been developed in~\cite{Gil-17}: it turns out that the CGF satisf{}ies a f{}irst order non-local partial differential equation that can be explicitly solved, thus giving access to many information such as  mean f{}itness (or mean trait since $\mathcal W(x)=x$), variance, position of the leading edge. In the present paper, we shall combine these two techniques to f{}irst reach an implicit expression of the mean f{}itness $\overline u(t)$ of the solution to~\eqref{eq:principale}, and then to obtain a so called {\it implicit/explicit} expression of the solution $u(t,x)$ itself. Additionally, this procedure allow us  to  develop  a simple  numerical  strategy  for  solving  the  Cauchy problem associated to~\eqref{eq:principale}.

In this work, and in contrast with~\eqref{modele}, we consider the case when the f{}itness function is modulated by $\overline{u}(t)=\overline x(t)$, the  mean-f{}itness at time $t$, as can be seen in \eqref{eq:principale}. Let us now make a short  description of the emergence of model \eqref{eq:principale} in \cite{Zad-17} to which we refer for further details.

The original  mutator model, yielding the aforementioned replicator-mutator model \eqref{modele}, is the continuous time model
\begin{equation}
\label{eq:mutator}
\frac{\partial u}{\partial t}(t,x)=(x-\overline x(t)) u(t,x),
\end{equation}
and considers the so-called {\it Malthusian fitness}, that is the rate of growth of a particular genotype.

On the other hand, in the case of non-overlapping generations, one may  start from the discrete time model 
\begin{equation}\label{discrete-time}
u(t+1,x)=\frac{xu(t,x)}{\overline x(t)}
\end{equation}
for the change in the so-called {\it Darwinian fitness} (success in propagating genes to the next generation) distribution. The Darwinian fitness being nonnegative, equation \eqref{discrete-time} is supplemented with $u(t=0,\cdot)$ supported in $[0,+\infty)$. 
Model \eqref{discrete-time} is immediately recast
\begin{equation}
\label{discrete-time2}
u(t+1,x)-u(t,x)=\frac{x-\overline x(t)}{\overline x(t)}u(t,x).
\end{equation}
Formally, at least for small times and narrow distributions, the above is approached by the continuous in time selection model
\begin{equation}
\label{approche}
\frac{\partial u}{\partial t}(t,x)=\frac{x-\overline x(t)}{\overline x(t)}u(t,x),
\end{equation}
which becomes \eqref{eq:principale} after inserting the effect of mutations through a diffusion term. Notice however that, in this derivation of \eqref{eq:principale}, the fitness of an individual with trait $x$ depends only on $x$ and thus selection is not density dependent.

In \cite{Zad-17}, the authors consider directed evolution, that is a laboratory technique that mimics natural evolution
and can be used for example to acquire proteins with new or improved properties. More precisely, they consider a  high-throughput compartmentalized  directed evolution process. Genotypes inside the compartments have dif{}ferent phenotypes. They not only pool their activity but also share the total number of produced copies, which makes the selection density dependent. In this process, the analogous of \eqref{discrete-time} is given by
\begin{equation}\label{discrete-time-bis}
u(t+1,x)=\frac{[g(l)x+(1-g(l))\overline x(t)]u(t,x)}{\overline x(t)}
\end{equation}
where the constant $g(+\infty)=0<g(l)<1=g(0)$ depends on $l$, a Poisson parameter measuring the occupancy of compartments. The analogous of \eqref{discrete-time2} is then
\begin{equation}
\label{discrete-time2-bis}
u(t+1,x)-u(t,x)=g(l)\frac{x-\overline x(t)}{\overline x(t)}u(t,x).
\end{equation}
Compared to \eqref{discrete-time2}, the presence of the coefficient $g(l)<1$ in \eqref{discrete-time2-bis} is the revelator of the density dependent selection. Now, if $g(l)$ is small (meaning $l$ large), the process is slowed down and the validity of \eqref{approche} as a continuous in time approximation should hold in much more cases than previously, meaning for larger time periods but also for more various shapes of distribution.

Hence, the compartmentalization that takes place in directed evolution (but also in natural systems like viruses with multiple infections) and the associated frequency dependence are the key tools to reach the continuous time model  \eqref{eq:principale}, starting from a discrete time model  written in terms of  the Darwinian fitness, see \cite{Zad-17}. 

Last, our second main focus, namely equation~\eqref{eq:v-equation}, corresponds to the replicator-mutator model \eqref{modele}, but with the additional effect that the mutations are frequency-dependent: the diffusion coef{}f{}icient is modulated by the mean trait $\overline u(t)=\overline x(t)$ and is thus \lq\lq socially determined''.

\section{Main results}\label{s:mainResults}

We here state our main results on~\eqref{eq:principale}, those  on~\eqref{eq:v-equation} will be gathered in the f{}inal section, where we transfer our developments on the solution $u$ of~\eqref{eq:principale} to the solution $v$ of~\eqref{eq:v-equation}.

\medskip

We f{}irst need to def{}ine an admissible set where to look after solutions. We denote by $\mathcal{A}$, the set of non-negative functions $f\in L^1(\R)$  such that\[\int _\R f(x)dx=1,\] 
which decrease faster than any exponential function, that is
\begin{equation}
    \label{eq:def-queues-legere}
    \lim _{x\to \pm \infty}f(x)\,e^{\alpha \vert x\vert}=0, \quad \forall \alpha>0,
\end{equation}
and, last, such that \[m_0:=\int _{\R} xf(x)dx >0.\]
\begin{rem}
Notice that the case $m_0<0$ follows by symmetry from the case $m_0>0$, whereas the case $m_0=0$ is singular, as clear from the equation and the Gaussian case studied in subsection~\ref{ss:gauss}. Notice also that assumption~\eqref{eq:def-queues-legere} could be relaxed by only assuming  limit $x\to +\infty$ when $m_0>0$ (or $x\to-\infty$ when $m_0<0$), the relevant tail being the right one as already observed in related situations by~\cite{Alf-Car-14} or~\cite{Gil-17}. This would require to adapt $(iii)$ in the below def{}inition by adding some integrability properties as $x\to -\infty$. For simplicity, in this work, we stick to~\eqref{eq:def-queues-legere}.
\end{rem}

\begin{defi}[Notion of solution]\label{def:sol}  Let $u_0\in \mathcal{A}$ be given. We say that $u=u(t,x)$ is a global solution of~\eqref{eq:principale} starting from $u_0$ if 
\begin{enumerate}[$(i)$]
    \item $u\in \mathscr C^{\infty}((0,+\infty)\times \R)$,
    \item for all $t\geq 0$, $u(t,\cdot)\in \mathcal{A}$,
    \item for all $t>0$, $\partial _t u(t,\cdot)$, $\partial _x u(t,\cdot)$ and $\partial _{xx}u(t,\cdot)$ decrease faster than any exponential function, in the sense of~\eqref{eq:def-queues-legere},
    \item $u$ solves~\eqref{eq:principale} in the classical sense,
    \item $u(t,\cdot)\to u_0$ in $L^1(\R)$, as $t\to 0$.
    \end{enumerate}
\end{defi}

Our main result consists in the well-posedness of the Cauchy problem~\eqref{eq:principale} with, moreover, an implicit/explicit expression of the solution.

\begin{thm}[The solution of the Cauchy problem]\label{th:main} Let $u_0\in \mathcal{A}$ be given. Then there is a unique global solution of~\eqref{eq:principale} starting from $u_0$ (in the sense of Def{}inition~\ref{def:sol}). Moreover, its mean $\overline u(t)$ is implicitly (and uniquely) determined by \[\overline u(t)=\sqrt{m_0^2+2\sigma^2t^2+2\int _0^t{C_0}''\left(\int_0^s\frac{dy}{\overline u(y)}\right)ds}\] where $m_0=\int_\R xu_0(x)dx>0$ and $C_0(z):=\ln\left( \int _\R u_0(x)e^{zx}dx\right)$, $z\geq 0$, is the cumulant generating function (\textsc{cgf}) of $u_0$. Last, $u(t,x)$ is given by \[ u(t,x) = w\left(t,x + 2\! \int_0^t\int_0^s \frac{\diffcst}{\overline u(\tau)} d\tau ds\right) \exp\left(-t+x\int_0^t  \frac{ ds }{\overline u(s)}
 +\int_0^t\!\! \diffcst \left(\int_0^s \frac{d\tau}{\overline u(\tau)} \right)^2 \! ds\right),\]
 where $w=w(t,y)$ is the solution of the Heat equation $\partial _t w=\sigma ^2\partial _{yy}w$ starting from $u_0$.
\end{thm}

The proof is based in a combination of the approaches of~\cite{Gil-17} and~\cite{Alf-Car-14}. In the course of the proof, we collect the expressions and some estimates on the mean $\bar u(t)$ and the variance
\begin{eqnarray*}
V(t)&:=&\int_{\R} (x-\bar u(t))^{2}u(t,x)dx\\
&=&
\int_\R x^2u(t,x)dx-\left(\int_\R x u(t,x)dx\right)^{2}
\end{eqnarray*}
of the solution at time $t$.

\begin{cor}[Long time behaviour]\label{cor:long} Let $u_0\in \mathcal{A}$ be given. Then the mean $\overline u(t)$ of the global solution given by Theorem~\ref{th:main} satisf{}ies
\[\overline{u}(t) \begin{cases}\geq \sqrt{2\diffcst}\, t &\text{ in any case}\\
\sim \sqrt{2\diffcst}\, t &\text{ if } \sup {\rm supp}(u_0)<+\infty,
\end{cases}\] and, in any case,
\begin{equation}\label{eq:not-int}
\int _0 ^{+\infty}\frac{dy}{\overline u (y)}=+\infty.
\end{equation}
The variance $V(t)$  of the global solution satisf{}ies \[V(t)\begin{cases}\geq 2\diffcst t &\text{ in any case}\\
\sim 2\diffcst t &\text{ if } \sup {\rm supp}(u_0)<+\infty,\\
\end{cases}\] where ${\rm supp}(u_0)$ denotes the support of $u_0$. 
\end{cor}

The above shows that, as $t\to+\infty$, the solution moves to the right since $\overline u(t)\to +\infty$, and is flattening since $V(t)\to +\infty$. In particular, for one side compactly supported initial data, the propagation (asymptotically) occurs at constant speed $\sqrt{2\diffcst}$, which is in contrast with the acceleration phenomenon proved in~\cite{Alf-Car-14}. This is also true for Gaussian data as will be noticed in subsection~\ref{ss:gauss}. The role of~\eqref{eq:not-int} is to provide \lq\lq a kind of upper bound'' on $\overline u(t)$ when, for example, the initial tails are still lighter than any exponential but heavier than Gaussian, e.g. of the magnitude $e^{-x^\alpha}$ ($1<\alpha<2$) or even $e^{-x \ln x}$ as $x\to +\infty$.

\medskip

The paper is organized as follows. In Section~\ref{s:special}, we investigate special solutions, in particular Gaussian ones which provide a preliminary understanding of equation~\eqref{eq:principale}. In Section~\ref{s:cgf}, we begin the proof of our main result, Theorem~\ref{th:main},  by using the \textsc{cgf} approach introduced in~\cite{Gil-17}. The proof of Theorem~\ref{th:main} and its corollary are completed in Section~\ref{s:explicit} thanks to the methodology of~\cite{Alf-Car-14}. As a by product of our analysis, we collect a numerical strategy for solving the Cauchy problem. Last, Section~\ref{s:v-section} is devoted to obtain the companion results on~\eqref{eq:v-equation} by expressing $v(t,x)$ in terms of $u(t,x)$.

\section{Special solutions}\label{s:special}

In this section, we investigate  special solutions to~\eqref{eq:principale}, in particular non-negative,  integrable steady states and Gaussian solutions. In the f{}irst case we prove the non existence of such steady states. This is due to the particular form of the f{}itness function and is an analogous result to that one obtained by Alfaro an Carles in~\cite{Alf-Car-14}. The situation is dif{}ferent in the case of a conf{}ining f{}itness function~\cite{Alf-Ver-18,Alf-Car-17} where the existence and uniqueness of a non-negative stationary solution is ensured and corresponds to the \emph{ground-state} of the Schr\"odinger Hamiltonian where the potential is the opposite of f{}itness function. In the second case, Gaussian solutions are computed explicitly by solving the dif{}ferential equations describing the evolution of the mean and the inverse of the variance.

\subsection{Steady states}\label{ss:steady}

\begin{prop}[Steady state]\label{prop:steady} 
There is no nontrivial non-negative steady state $\phi=\phi(x)$ solving~\eqref{eq:principale} and satisfying $\phi(\pm \infty)=0$.
\end{prop}

\begin{proof} A steady state with $\overline x\neq 0$ must solve $\phi''(x)+\frac 1 \diffcst \frac{x-\overline x}{\overline x}\phi(x)=0$ for all $x\in \R$. Letting $\psi(x):=\phi(\overline x- {(\diffcst \overline x)}^{1/3}\, x)$, this is recast \[\psi''(x)-x\psi(x)=0, \quad x\in \R.\] Hence $\psi(x)$ is a linear combination of the Airy functions ${\rm Ai}(x)$ and ${\rm Bi}(x)$. From $\psi(+\infty)=0$, we deduce that $\psi(x)$ is a multiple of ${\rm Ai}(x)$. Hence either it is trivial, or it changes sign.
\end{proof}

\subsection{Gaussian solutions}\label{ss:gauss}

We  investigate the propagation of a Gaussian initial data, which is relevant for biological applications. In the context of evolutionary genetics, families of Gaussian solutions for nonlinear and nonlocal equations can be found in~\cite{Biktashev2014}, \cite{Alf-Car-14, Alf-Car-17}. In a dif{}ferent context involving logarithmic non-linearities, we also refer to~\cite{BiMy76}, \cite{CaGa-p} for
the Schr\"odinger equation and to~\cite{Alf-Car-18} for the Heat equation.
 
\begin{prop}[Propagation of Gaussian initial data]\label{prop:prop-gauss} For $a_0>0$ and $m_0\in \R$, let us def{}ine
\begin{equation}\label{eq:a-et-m}
a(t):=\frac {a_0}{1+2a_0\diffcst t}, \quad m(t):=\begin{cases}+\sqrt{2\diffcst t^{2}+\frac{2t}{a_0}+m_0^{2}}& \text{ when } m_0>0\\
\pm \sqrt{2\diffcst t^{2}+\frac{2t}{a_0}} & \text{ when } m_0=0\\
-\sqrt{2\diffcst t^{2}+\frac{2t}{a_0}+m_0^{2}} & \text{ when } m_0<0.
\end{cases} 
\end{equation}
Then
\begin{equation}
\label{eq:gauss-u}
u(t,x):=\sqrt{\frac {a(t)}{2\pi}}e^{-a(t)\frac{(x-m(t))^2}2}
\end{equation}
 solves~\eqref{eq:principale} in $(0,+\infty)\times \R$ and starts from $u_0(x)=\sqrt{\frac {a_0}{2\pi}}e^{-a_0\frac{(x-m_0)^2}2}$.
\end{prop}

\begin{proof} {}From the ansatz~\eqref{eq:gauss-u}, we compute
\begin{eqnarray*}
\partial _t u(t,x)&=&u(t,x)\left[\frac{a'(t)}{2a(t)}-a'(t)\frac{(x-m(t))^{2}}{2}+a(t)m'(t)(x-m(t))\right]\\
\partial_{xx} u(t,x)&=&u(t,x)\left[-a(t)+a^{2}(t)(x-m(t))^2\right]\\
u(t,x)\frac{x-\overline x(t)}{\overline x(t)}&=&u(t,x)\frac{x-m(t)}{m(t)}.
\end{eqnarray*}
We plug the above into equation~\eqref{eq:principale}, and identify the $x^{2}$, the $x^{1}$ and the $x^0$ coef{}f{}icients to obtain three  dif{}ferential equations. The f{}irst one is
\begin{equation}\label{eq:eq-sur-a}
-\frac{a'(t)}2=\diffcst a^{2}(t), 
\end{equation}
whose solution, starting from $a_0$, is given by the f{}irst part in~\eqref{eq:a-et-m}. Using~\eqref{eq:eq-sur-a}, we see that the two other equations reduce to $a(t)m(t)m'(t)=1$, which is solved as \[m^{2}(t)=2\diffcst t^{2}+\frac{2t}{a_0}+m_0^{2},\] and thus the second part in~\eqref{eq:a-et-m}.
\end{proof}

Notice that functions $m(t)$ and $V(t)=\frac{1}{a(t)}$ are respectively the mean and the variance of the density $u(t,\cdot)$. Hence Proposition~\ref{prop:prop-gauss} shows that, starting from a Gaussian prof{}ile, the solution remains a Gaussian function, is asymptotically propagating at constant speed and flattening since
$m(t)\sim \pm \sqrt{2\diffcst}\, t$, $V(t)\sim 2\diffcst t$, as $t\to +\infty$, see Figure~\ref{fig:gauss-centree}. Notice that the direction of propagation is determined by the initial value $m_0$: towards the right when $m_0>0$, the left when $m_0<0$. The case $m_0=0$ is singular because of the multiplicity of solutions to~\eqref{eq:a-et-m}, two Gaussian solutions emerge, propagating to the left and to the right, see Figure~\ref{fig:gauss-centree}.

\begin{figure}[ht]
  \includegraphics[width=1\textwidth]{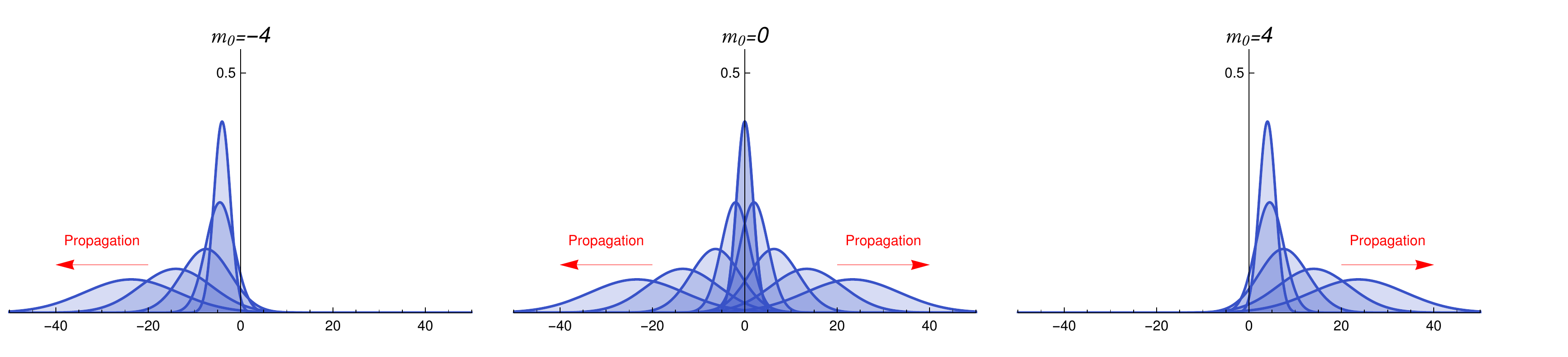}
    \caption{Evolution of Gaussian solutions for $\diffcst=1$,  $a_0=1$ and (from left to right) $m_0=-4$, $m_0=0$ and $m_0=4$.}
    \label{fig:gauss-centree} 
\end{figure}

\section{The  \textsc{CGF} approach}\label{s:cgf}

In this section, we assume that we are equipped with a solution of~\eqref{eq:principale} starting from $u_0\in \mathcal{A}$.  We def{}ine the  Cumulative Generating Function (\textsc{cgf}) 
\begin{equation}
    \label{eq:CGF-v-definition}
    C(t,z):= \ln\left(\int_\R u(t,x) e^{z x}\, dx  \right), \quad t\geq 0, z\geq 0,
\end{equation}
of such a solution, in the spirit of~\cite{Gil-17}.  From Def{}inition~\ref{def:sol}, $C(t,z)$ is well def{}ined and smooth on $(0,+\infty)\times [0,+\infty)$. We shall derive an implicit expression for $C(t,z)$, and then for the value $\overline u(t)$ of the mean. This crucial step will enable us to complete the analysis in the next section, which is much in the spirit of~\cite{Alf-Car-14}.

\medskip

The following computations are validated by the notion of solution adopted in Def{}inition~\ref{def:sol} and, possibly, the dominated convergence theorem. Observe that 
$$
\partial _t C(t,z)=\frac{\int _\R \partial _t u(t,x) e^{zx}dx}{\int _\R u(t,x)e^{zx}dx},
$$
and that
$$\partial _z C(t,z)=\frac{\int _\R x u(t,x) e^{zx}dx}{\int _\R u(t,x)e^{zx}dx},\quad \partial _{zz} C(t,z)=\frac{\int _\R x^{2} u(t,x) e^{zx}dx}{\int _\R u(t,x)e^{zx}dx}-\left(\frac{\int _\R x u(t,x) e^{zx}dx}{\int _\R u(t,x)e^{zx}dx}\right)^{2}.
$$
Notice in particular that the mean is reached through
$$
\partial _z C(t,z=0)=\int _\R xu(t,x)dx=\overline{u}(t),
$$
whereas the variance is reached through
$$
\partial _{zz} C(t,z=0)=\int _\R x^{2} u(t,x) dx-\left(\int _\R x u(t,x) dx\right)^{2}=V(t).
$$

Hence, multiplying equation~\eqref{eq:principale} by $e^{zx}$, integrating over $x\in\R$ and f{}inally dividing by $\int_\R u(t,x)e^{z x}$, we obtain the following nonlocal f{}irst order \textsc{pde}
\begin{equation}
\label{eq:CGFequation}
     \partial_t C(t,z) = \diffcst z^{2}-1+\frac{\partial _z C(t,z)}{\partial _z C(t,0)},
\end{equation}
where we have used integration by parts to get \[\int _\R \partial_{xx} u(t,x)e^{zx}dx=z^{2}\int _\R u(t,x)e^{zx}dx.\] Furthermore, since $\int_\R u(t,x) dx = 1$, the condition $C(t,0)=0$ must hold for any $t\geq0$. 
As a result we are facing the problem
\begin{equation}\label{eq:Cauchy-CGF}
    \begin{cases}
     \partial_t C(t,z) = \diffcst z^{2}-1+\displaystyle{\frac{\partial _z C(t,z)}{\overline{u} (t)}} &t\geq 0, z\geq 0,\\
    C(0,z) =C_0(z) &z\geq 0, \vspace{5pt}\\
    C(t,0)=0 &t\geq 0,
    \end{cases}
\end{equation}
where $C_0(z)=\ln\left(\int _\R u_0(x)e^{zx}dx\right)$ is the \textsc{cgf} of the initial data $u_0$. Notice that, from a straightforward computation, \[{C_0}''(z)=\frac{\left(\int _\R x^2u_0(x)e^{zx}dx\right)\left(\int _\R u_0(x)e^{zx}dx\right)-\left(\int _\R xu_0(x)e^{zx}dx\right)^{2}}{\left(\int _\R u_0(x)e^{zx}dx\right)^{2}}.\]
Hence, from Cauchy Schwarz inequality, we see that ${C_0}''(z)\geq 0$ for all $z\geq 0$, and even ${C_0}''(z)>0$ for all $z\geq 0$ if $u_0\not \equiv 0$. This convexity property of cumulant generating functions is well-known in probability, and will be used in the following.
\medskip

Fix $t> 0$ and $z> 0$. For  $s$ such that $-t \leq s$ and $\int_ 0^s \frac{ d\tau}{\overline{u}(t+\tau)}\leq z$, set \[\psi(s):=C\left(t+s,z-\int _0 ^s \frac{d\tau}{\overline{u}(t+\tau)} \right),\] which we dif{}ferentiate to get
\begin{eqnarray}
    \psi'(s)  &=& \partial_t C\left(t+s,z-\int _0 ^s \frac{ d\tau}{\overline{u}(t+\tau)}\right)-\frac{1}{\overline{u}(t+s)}\partial_z C\left(t+s,z-\int _0 ^s \frac{ d\tau}{\overline{u}(t+\tau)}\right)\nonumber\\
    &=&\diffcst \left(z-\int _0 ^s \frac{ d\tau}{\overline{u}(t+\tau)}\right)^{2}-1.\label{bidule}
\end{eqnarray}
As a result,
\begin{eqnarray*}
    C(t,z)-C_0 \left(z+\int _0^t \frac {d\tau}{\overline{u} (\tau)}\right)
   & =& C(t,z)-C_0\left(z+\int_{-t}^{0} \frac{ d\tau}{\overline{u}(t+\tau)}\right)\\& =&\psi(0)-\psi(-t)\\
    & =& \int_{-t}^0 \psi'(s)\, ds. 
    \end{eqnarray*}
From~\eqref{bidule}, we deduce that
    \begin{eqnarray*}
 &&C(t,z)-C_0 \left(z+\int _0^t \frac {d\tau}{\overline{u} (\tau)}\right)\\
 &&\quad = \diffcst \int _{-t}^0\left(z-{\int _0^s \frac{ d\tau}{\overline{u}(t+\tau)}}\right)^{2}\, ds-t\\
   &&\quad = \diffcst tz^2-2\diffcst z \int_{-t}^0{ \int _0^s \frac{ d\tau}{\overline{u}(t+\tau)}} ds+\diffcst \int_{-t}^0\left({\int_0^s\frac{ d\tau}{\overline{u}(t+\tau)}}\right)^{2}\,ds-t\\
    &&\quad = \diffcst tz^2+2\diffcst z \int _{-t}^0{\int _{-t}^\tau \frac{1}{\overline{u}(t+\tau)}\,ds}d\tau +\diffcst \int_{-t}^0\left({\int_0^s\frac{ d\tau}{\overline{u}(t+\tau)}}\right)^{2}\,ds-t,
    \end{eqnarray*}
  by Fubini-Tonelli theorem. We thus conclude that  
     \begin{eqnarray*}
     C(t,z)=C_0 \left(z+\int _0^t \frac {d\tau}{\overline{u} (\tau)}\right) + \diffcst tz^2+2\diffcst z \int_{0}^t \frac{y}{\overline{u}(y)}\,dy +\diffcst \int
       _{-t}^0\left({\int_0^s\frac{ d\tau}{\overline{u}(t+\tau)}}\right)^{2}\,ds-t.
\end{eqnarray*}
This is an implicit expression for $C(t,z)$ since it still involves $\overline{u}(t)=\partial_z C(t,0)$. Differentiating with respect to $z$ and evaluating at $z=0$, we reach another implicit formula for the mean value of the solution 
\begin{equation}
\label{eq:implicit-ubar}
\overline{u}(t)={C_0}'\left(\int _0^{t}\frac{dy}{\overline{u}(y)}\right)+2\diffcst  \int _{0}^t \frac{y}{\overline{u}(y)}\,dy,
\end{equation}
whereas  dif{}ferentiating twice with respect to $z$ and evaluating at $z=0$, we reach the variance of the solution
\begin{equation}
\label{eq:implicit-var}
V(t)={C_0}''\left(\int _0 ^t \frac{dy}{\overline{u}(y)}\right)+2\diffcst t\begin{cases}
\geq 2\diffcst t &\text{ in any case}\\
\sim 2\diffcst t &\text{ if } \sup {\rm supp}(u_0)<+\infty,
\end{cases}
\end{equation}
where the last estimate follows from~\cite[Lemma 4.5]{Gil-17}.

\medskip

Next, dif{}ferentiating~\eqref{eq:implicit-ubar} with respect to $t$, we get
\[{\overline{u}}\,'(t)=\frac{1}{\overline{u}(t)}{C_0}''\left(\int _0 ^t \frac{dy}{\overline{u}(y)}\right)+\frac{2\diffcst t}{\overline{u}(t)},\]
so that $\frac{d}{dt}\overline{u} ^2(t)=2V(t)$ and thus
\begin{equation}
\label{eq:asy-ubar}
\overline{u}(t) \begin{cases}\geq \sqrt{2\diffcst}\, t &\text{ in any case}\\
\sim \sqrt{2\diffcst}\, t &\text{ if } \sup {\rm supp}(u_0)<+\infty.
\end{cases}
\end{equation}
Also, integrating  $\frac{d}{dt}\overline{u} ^2(t)=2V(t)$, we collect the very useful expression
\begin{equation}
\label{eq:ubar-egal}
\overline{u}(t)=\sqrt{m_0^2+2\diffcst t^2+2\int_0^t {C_0}''\left(\int_0^s\frac{dy}{\overline{u} (y)}\right)\,ds}.
\end{equation}

\begin{rem}
Notice that if $u_0(x)=\sqrt{\frac{a_0}{2\pi}}e^{-a_0\frac{(x-m_0)^{2}}{2}}$ is a Gaussian initial data as in subsection~\ref{ss:gauss}, we know that its \textsc{cgf} is \[C_0(z)=m_0z+\frac 1{2a_0}z^2.\] Hence $\overline{u}(t)=\sqrt{m_0^2+2\diffcst t^2+\frac{2}{a_0}t}$  and $V(t)=\frac{1}{a_0}+2\diffcst t$, which agrees with the values of $m(t)$ and $\frac{1}{a(t)}$ in Proposition~\ref{prop:prop-gauss}. 
\end{rem}

\section{The solution implicitly/explicitly}\label{s:explicit}

In this section, we mainly complete the proof of Theorem~\ref{th:main} and of Corollary~\ref{cor:long}. As a by product, we present some numerical strategies for solving~\eqref{eq:principale}. 

\medskip
First, assume that we are equipped with a solution $u(t,x)$ of~\eqref{eq:principale} starting from $u_0\in \mathcal{A}$.   In particular we know from Section~\ref{s:cgf} that $\overline{u}(t)$ is given by~\eqref{eq:ubar-egal}. Following~\cite{Alf-Car-14},  we write
\begin{equation}
\label{eq:def-w}
   u(t,x) = w\left(t,x + 2 \int_0^t\int_0^s \frac{\diffcst}{\overline u(\tau)} d\tau ds\right)e^{-t+x\int_0^t  \frac{ ds }{\overline u(s)}
 +\int_0^t \diffcst \left(\int_0^s \frac{d\tau}{\overline u(\tau)} \right)^2  ds}.
\end{equation}
This clearly def{}ines a unique $w(t,y)$, $t\geq 0$, $y\in \R$, which has to satisfy
\begin{equation}
\label{eq:w-equation}
\begin{cases}
\partial _t w=\diffcst \partial _{yy} w\\ w_{\mid t=0}=u_0.
\end{cases}
\end{equation}
We refer to~\cite{Alf-Car-14} for more details on such a change of unknown function, which enables to reduce some replicator-mutator equations to the Heat equation.

Now, our main task is to show that the problem~\eqref{eq:ubar-egal}, is globally well-posed.

\begin{lem}\label{lem:pointFixe}
For any given $T>0$, there is a unique solution $\overline{u}\in\mathscr{C}^0([0,T])$ to \begin{equation}\label{eq:PointFixeEquation}
    \overline{u}(t)=\sqrt{m_0^2+2\diffcst t^2+2\int_0^t {C_0}''\left(\int_0^s\frac{dy}{\overline{u} (y)}\right)\,ds}.
\end{equation}
\end{lem}

\begin{proof}
Recall that $m_0>0$. We def{}ine the following subspace of $\mathscr{C}^0([0,T])$, equipped with the $L^\infty$ norm, \[X:=\{h\in \mathscr{C}^0([0,T]), h(t)\geq h(0)=m_0 \; \text{ for all } 0\leq t\leq T\},\]
which is a Banach space. Now, we def{}ine the operator $\mathcal L: h\in X\mapsto q \in X$, and $q$ is given by \[q(t):=\sqrt{m_0^2+2\sigma^2 t^2 + 2 \int_0^t C_0{''}\left(\int_0^s \frac{d\tau}{h(\tau)} \right)\, ds}.\]

The proof consists in mimicking that of the Banach f{}ixed point theorem. In order to prove existence, we introduce the sequence $(q_n)\in X^{\N}$ def{}ined inductively by 
\begin{equation}\label{eq:iterationsquence}
\begin{cases}
q_0(t)=m_0\\
q_{n+1}(t)=\sqrt{m_0^2+2\sigma^2 t^2 + 2 \int_0^t C_0{''}\left(\int_0^s \frac{d\tau}{q_n(\tau)} \right)\, ds}.
\end{cases}
\end{equation}
We have
\begin{align*}
    \abs{q_{n+1}(t)-q_n(t)} &=  \dfrac{\abs{q_{n+1}^2(t)-q_n^2(t)}}{q_{n+1}(t)+q_n(t)}\\
    & \leq \frac{1}{m_0} \int_0 ^t \abs{C_0{''}\left( \int_0^s \frac{d\tau}{q_n(\tau)} \right) - C_0{''}\left( \int_0^s \frac{d\tau}{q_{n-1}(\tau)} \right)}  \, ds\\
    & \leq \frac{1}{m_0}{\lVert{{C_0}^{(3)}}\rVert}_{L^\infty(0,T)} \int_0^t \abs{\int_0^s \left( \frac{1}{q_n(\tau)}-\frac{1}{q_{n-1}(\tau)}\right) d\tau}  \, ds\\
    & \leq \frac{1}{m_0}{\lVert{{C_0}^{(3)}}\rVert}_{L^\infty(0,T)} \int_0^t  \int_0^s \frac{\abs{q_n(\tau)-q_{n-1}(\tau)}}{m_0^2}   \, d\tau\,ds\\
    & \leq \frac{1}{m_0^3}{\lVert{{C_0}^{(3)}}\rVert}_{L^\infty(0,T)} \int_0^t  \int_0^s \norme{q_n-q_{n-1}}_{L^\infty(0,s)}   \, d\tau\,ds\\
    & \leq \frac{1}{m_0^3}{\lVert{{C_0}^{(3)}}\rVert}_{L^\infty(0,T)} \int_0^t s\  \norme{q_n-q_{n-1}}_{L^\infty(0,s)}  \,ds\\
    & \leq k\, T \int_0^t \norme{q_n-q_{n-1}}_{L^\infty(0,s)} \,ds,
\end{align*}
 where $k:=\frac{1}{m_0^3} {\lVert{{C_0}^{(3)}}\rVert}_{L^\infty(0,T)}$. Then we straightforwardly prove by induction that, for any $n\geq 0$, any $0\leq s\leq T$,
 \begin{equation}\label{eq:propFactorial}
 {\norme{q_{n+1}-q_n}}_{L^\infty(0,s)} \leq {\norme{q_{1}-q_0}}_{L^\infty(0,T)}\frac{(k\, Ts)^n}{n!}.
 \end{equation}
 In particular, the series  $\sum {\norme{q_{n+1}-q_n}}_{L^\infty(0,T)}$ converges and therefore  $(q_n)$ converges uniformly in $\mathscr{C}^0(0,T)$ to some $q$ which is a f{}ixed point of $\mathcal L$.

As far as uniqueness is concerned, if $q$ and $\tilde q$ are two f{}ixed points of $\mathcal L$ then the argument to reach~\eqref{eq:propFactorial} also yields
$\Vert q - \tilde q\Vert _{L^\infty(0,T)}\leq \Vert q - \tilde q\Vert _{L^\infty(0,T)} \frac{(kT^{2})^{n}}{n!}$. Letting $n\to +\infty$ enforces $q\equiv \tilde q$.
\end{proof}

\begin{proof}[Completion of the proof of Theorem~\ref{th:main} and Corollary~\ref{cor:long}] Hence, any solution in the sense of Def{}inition~\ref{def:sol} has to be given by~\eqref{eq:def-w}, where $\overline u(t)$ is given by Lemma~\ref{lem:pointFixe}. Notice that, from~\eqref{eq:PointFixeEquation}, $\overline u$ has to be smooth.

Conversely, it is now a matter of straightforward computations --- based on~\eqref{eq:def-w},\eqref{eq:w-equation} and~\eqref{eq:PointFixeEquation}--- to check that this does provide the solution. Notice that, in particular, the initial data is understood in the sense of Def{}inition~\ref{def:sol} $(v)$ since, in the  Cauchy problem for the heat equation~\eqref{eq:w-equation}, the initial data is understood in the sense $w(t,\cdot)\to u_0$ in $L^1(\R)$ as $t\to 0$. This proves  Theorem~\ref{th:main}.

Last, the conclusions of Corollary~\ref{cor:long} have been collected in the course of Section~\ref{s:cgf}, except estimate~\eqref{eq:not-int} which will be obtained in the proof of Lemma~\ref{lem:ODE}, where its relevance will become much clearer.
\end{proof}

\noindent {\bf Numerical implications.}  Three  major dif{}f{}iculties  are encountered when setting up a numerical strategy for problem~\eqref{eq:principale}. The f{}irst one is the nonlocal nature of the equation. At this stage, two natural approaches can be considered: $i)$ use f{}inite differences and approximate the nonlocal term by a Riemann sum; $ii)$ apply a splitting method, computing alternatively the solution of the non-local term and the resulting local reaction-diffusion equation. 

The second dif{}f{}iculty is the unboundedness of the domain. To manage it, one usually imposes artificial boundary conditions and solves numerically the equation on a suf{}f{}iciently large domain, approximating the true solution by the latter. 

The third complexity comes from the propagation of the solution, at least linear in view of~\eqref{eq:asy-ubar}, making it dif{}f{}icult to track it over time.

\medskip

Previously, we gave an implicit/explicit construction of the solution $u=u(t,x)$ of the Cauchy problem associated with equation \eqref{eq:principale}, which actually provides a strategy for solving numerically the problem. 

The idea is to initially f{}ind  an approximation of the nonlocal term $\overline{u}(t)$ on a time interval $[0,T],\  T>0$. This step is performed using the f{}ixed-point iteration sequence~\eqref{eq:iterationsquence}, for which the estimate~\eqref{eq:propFactorial} provides a convergence rate. This requires in particular to  compute (analytically or numerically)  the cumulant generating function $C_0(z)$  of the initial datum $u_0(x)$.   

 The next step consists in calculating the solution $w=w(t,x)$ of the heat equation with initial datum $u_0(x)$ and evaluating it in the values indicated by formula~\eqref{eq:def-w}.  Obviously, in the case where an explicit solution $w=w(t,x)$ of the heat equation can be found, artificial boundary conditions for solving the heat equation are not needed, leading to a better numerical solution. On the other hand, in the case where no closed form for $w=w(t,x)$ is available, it is important to previously solve numerically the heat equation in a suf{}f{}iciently large spatial domain. For instance, if we  want to solve numerically~\eqref{eq:principale} in $[x_-,x_+]\times [0,T]\subset \R\times \R$, assuming $m_0>0$ (the case $m_0<0$ following by symmetry), then it is clear from~\eqref{eq:def-w} that the function $w$ needs to be evaluated in space up to $x=x_++ 2 \int_0^t\int_0^s \frac{\diffcst}{\overline u(\tau)} d\tau ds$ and since $\overline{u}(t)\geq m_0>0$, then  we should solve the heat equation at most in $[x_-,x_++\tfrac{\diffcst T^2}{m_0}]\times[0,T]$. 

\medskip

In f{}igure~\ref{fig:numericalsolution}, we have represented the numerical solution obtained by this method for the initial condition $u_0(x)=\mathds{1}_{[\tfrac{1}{2},\tfrac{3}{2}]}(x)$, so that $m_0=1$ and $C_0(z)=\ln\left(\frac{2 e^z \sinh(z/2)}{z} \right)$,
 and $\diffcst=1$;  in this case, $w$ is known to be given by  \[w(t,x)=\frac{1}{2} \left(\text{erf}\left(\frac{3-2 x}{4
   \sqrt{t}}\right)-\text{erf}\left(\frac{1-2 x}{4
   \sqrt{t}}\right)\right).\]

\begin{figure}
  \begin{subfigure}[b]{0.45\textwidth}
  \centering
    \includegraphics[scale=0.23]{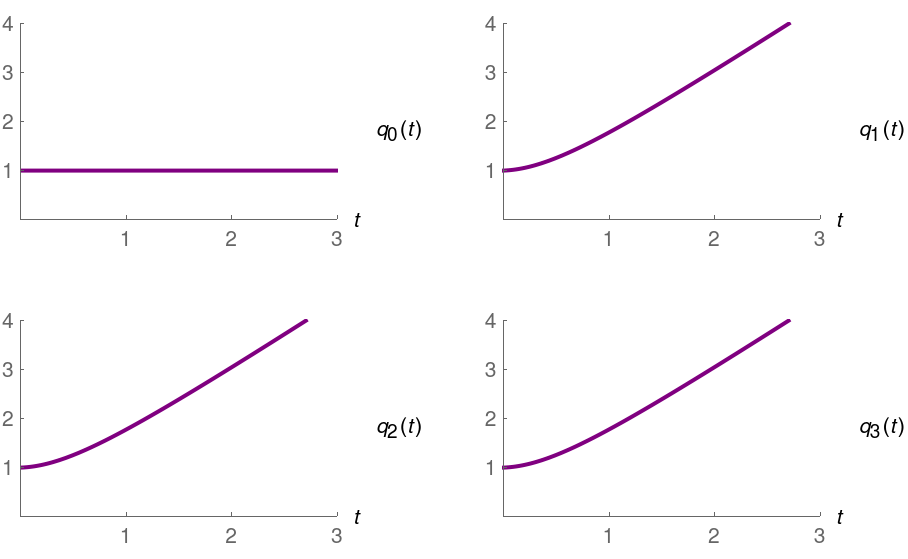}
    \caption{}
    \label{fig:2}
  \end{subfigure}
  \begin{subfigure}[b]{0.45\textwidth}
    \qquad \includegraphics[width=\textwidth]{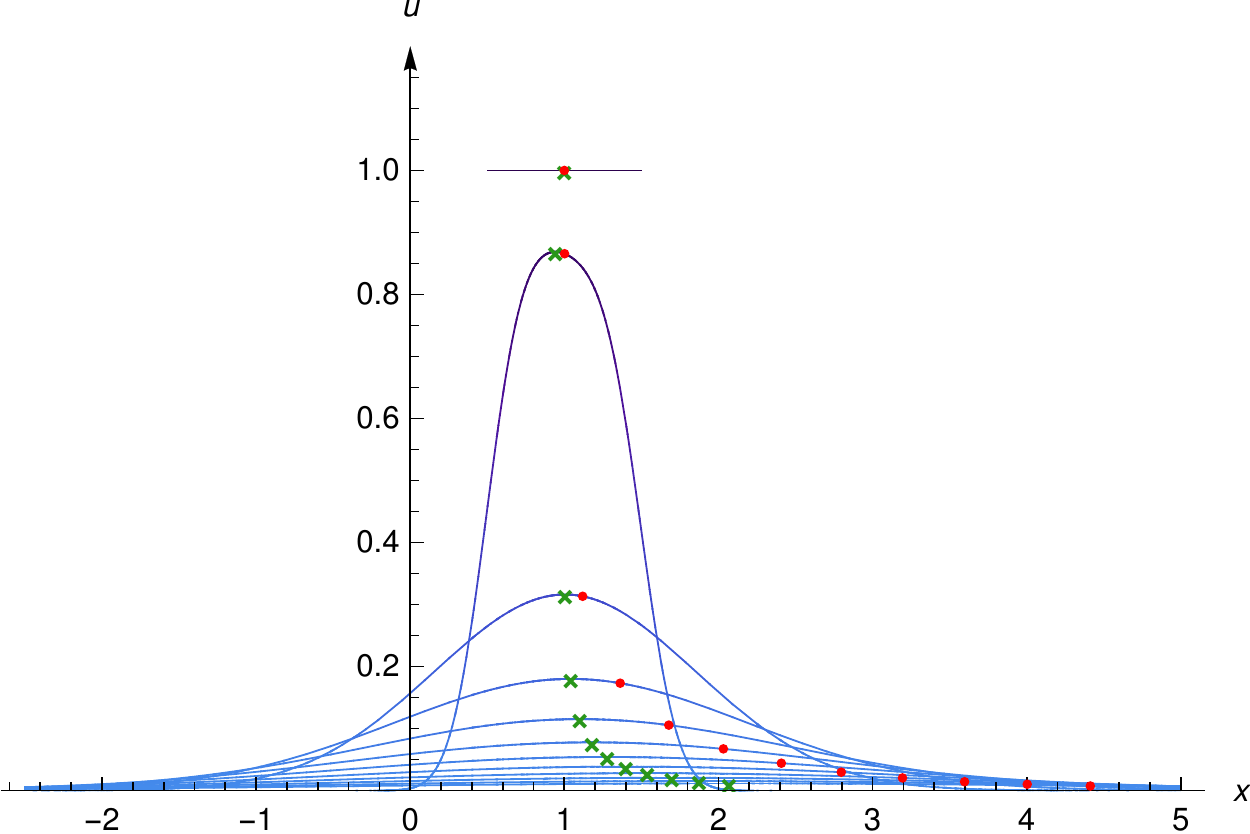}
    \caption{}
    \label{fig:numericalsolution}
  \end{subfigure}
\caption{(a): The f{}irst four approximations, for $0\leq t\leq 3$, of the nonlocal term $\overline{u}(t)$ computed via the f{}ixed point iteration~\eqref{eq:iterationsquence}. (b): Numerical solution obtained by the method described in Section~\ref{s:explicit}, starting from $u_0=\mathds{1}_{[1/2,3/2]}$, with $\diffcst=1$. The red points are the points on the graph $u(t,\cdot)$ with abscissa $x=\overline u(t)$.
The green points are the maxima of $u(t,\cdot)$. This reveals the dissymmetry of the solution.}
\end{figure}

\section{When the mutation coef{}f{}icient is socially determined}
\label{s:v-section}

This section is devoted to the analysis of equation~\eqref{eq:v-equation} supplemented with a non-negative initial data $v_0\in L^1(\R)$ satisfying $\int _\R v_0(x)\, dx=1$, and decreasing faster than any exponential data in the sense of~\eqref{eq:def-queues-legere}. We denote $m_0:=\int _\R xv_0(x)\,dx$. Interestingly, one can avoid the assumption $m_0>0$  in subsection~\ref{ss:gauss-v} and obtain anti-dif{}fusing/dif{}fusing solutions that are mathematically interesting. 

\subsection{Gaussian solutions: concentration vs. extinction}\label{ss:gauss-v}
 
\begin{prop}[Propagation of Gaussian initial data]\label{prop:prop-gauss-v} For $a_0>0$ and $m_0\in \R$. Then there is a unique Gaussian
\begin{equation}
\label{eq:gauss-v}
v(t,x):=\sqrt{\frac {a(t)}{2\pi}}e^{-a(t)\frac{(x-m(t))^2}2}
\end{equation}
which solves, at least locally in time,~\eqref{eq:v-equation} and starts from $u_0(x)=\sqrt{\frac {a_0}{2\pi}}e^{-a_0\frac{(x-m_0)^2}2}$. Its behaviour depends strongly on the parameters $a_0>0$, $m_0\in \R$ and $\sigma>0$.
\begin{enumerate}[$(i)$]
\item Assume $m_0<-\frac{1}{a_0\sqrt{2\diffcst}}$. Then $a(t)$ blows up in f{}inite time: there is  $0<T^\star<+\infty$ such that \[m(t)\to m(T^{\star})<0 \quad\text{ and }\quad  V(t):=\frac{1}{a(t)}\to 0, \text{ as } t\nearrow T^{\star}.\]

\item Assume $m_0=-\frac{1}{a_0\sqrt{2\diffcst}}$. Then $a(t)$ and $m(t)$ are global and \[
m(t)\nearrow 0 \quad\text{ and }\quad  V(t):=\frac{1}{a(t)}\to 0, \text{ as } t\to +\infty.\]

\item Assume $m_0>-\frac{1}{a_0\sqrt{2\diffcst}}$. Then $a(t)$ and $m(t)$ are global and \[
m(t)\sim C_me^{\sqrt{2\diffcst}t}\to +\infty \quad\text{ and }\quad  V(t):=\frac{1}{a(t)}\sim C_V e^{\sqrt{2\diffcst}t}\to +\infty, \text{ as } t\to +\infty,\] where $C_m:=\frac{2a_0 m_0 \sqrt{\diffcst}+\sqrt{2}}{4a_0\sqrt{\diffcst}}>0$ and $C_V:=\frac{2a_0 m_0 \sqrt{\diffcst}+\sqrt{2}}{2\sqrt 2 a_0}>0$.
\end{enumerate}
\end{prop}

\begin{rem} In Figure~\ref{fig:RegionsVbarLimit}, we have subdivided the phase plane $(m_0,a_0)$ into regions corresponding to the cases $(i),(ii)$ and $(iii)$.

Case $(i)$  corresponds to a concentration phenomena in f{}inite time:  the Gaussian solution converges, in f{}inite time, to a Dirac mass centered at  $x=m(T^\star)<0$. See Figure~\ref{fig:concentrationFiniteTime}.

Case $(ii)$  corresponds to a concentration phenomena in inf{}inite time:  the Gaussian solution converges, at large times, to a Dirac mass centered at  $x=0$. See Figure~\ref{fig:concentrationInfiniteTime}.

In case $(iii)$, if $m_0<0$ then, in contrast with cases $(i)$ and $(ii)$, the mean of the solution manages to \lq\lq cross'' the value zero. In case $(iii)$ convergence to a accelerating ($m(t)\sim C_m e^{\sqrt{2\diffcst}t}$) and flattening ($V(t)\sim C_Ve^{\sqrt{2\diffcst}t}$) prof{}ile always occurs at large times. Moreover, the variance of the solution is decreasing before the mean reaches zero, and then starts to increase after the mean crosses zero, which reveals an anti-dif{}fusion/dif{}fusion phenomenon, see Figure~\ref{fig:focusinDefocusing}. On the other hand, if $m_0\geq0$, the solution does not \emph{anti-dif{}fuse} and only dif{}fuses while the mean tends to inf{}inity, see Figure~\ref{fig:accelerating}.
\end{rem}

\begin{figure}
    \centering
    \includegraphics[width=0.5\textwidth]{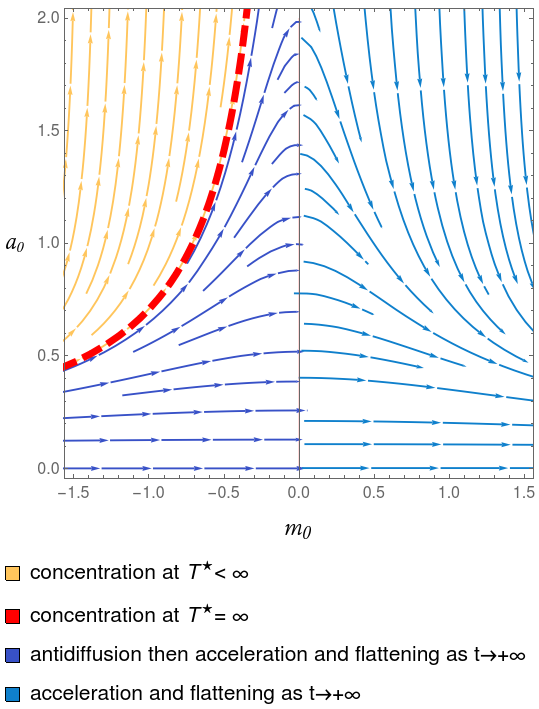}
    \caption{Vector f{}ield def{}ined by the dif{}ferential system~\eqref{eq:Zad-meanLaplacian} with $\diffcst=1$, describing the dynamics of Gaussian solutions. In yellow, the set of initial conditions for which $a$ blows up in f{}inite time $T^{\star}$ and  in red, dark blue and light blue those for which both $a$ and $m$ are globally def{}ined. The red dashed curve is the set of values  def{}ined by $m_0=-1/{(a_0\sqrt{2\diffcst})}$, for which $a$ tends to inf{}inity and $m$ tends to zero as time goes to inf{}inity. The dark blue region corresponds to the values leading to an anti-dif{}fusion/dif{}fusion behaviour. The light blue region corresponds to the values leading to a pure dif{}fusion behaviour.}
    \label{fig:RegionsVbarLimit}
\end{figure}

\begin{proof} As in the proof of Proposition~\ref{prop:prop-gauss}, we plug the ansatz~\eqref{eq:gauss-v} into~\eqref{eq:v-equation} and arrive at
\begin{equation}\label{eq:Zad-meanLaplacian}
    \begin{cases}
    m'(t) & = \frac{1}{a(t) }\\
    a'(t) & = -2 \diffcst a(t)^2m(t),
    \end{cases}
\end{equation}
so that $m''(t)=2\diffcst m(t)$, $m(0)=m_0$, $m'(0)=\frac{1}{a_0}$ which is globally solved as 
\begin{equation}\label{eq:sol-m-v}
    m(t)=\frac{\sqrt{2}+2a_0 m_0 \sqrt{\diffcst}}{4a_0\sqrt{\diffcst}}e^{\sqrt{2\diffcst}t}-\frac{\sqrt{2}-2a_0 m_0 \sqrt{\diffcst}}{4a_0\sqrt{\diffcst}}e^{-\sqrt{2\diffcst}t}.
\end{equation}
From the f{}irst equation in~\eqref{eq:Zad-meanLaplacian} we deduce 
\begin{equation}\label{eq:sol-a-v}
 a(t)=\frac{2\sqrt{2}\, a_0 }{(\sqrt{2}+2 a_0 m_0 \sqrt{\diffcst})e^{\sqrt{2\diffcst} t}+(\sqrt{2}-2 a_0 m_0\sqrt{\diffcst})e^{-\sqrt{2\diffcst} t}},
\end{equation}
as long as blow-up does not occur. We now easily see that in case $(i)$ blow up occurs at time \[T^\star=\frac{1}{2\sqrt{2\diffcst}}\ln\left(\frac{2  a_0 m_0 \sqrt{\diffcst}-\sqrt{2}}{2  a_0 m_0 \sqrt{\diffcst}+\sqrt{2}}\right)\in(0,+\infty).\] In case $(ii)$ blow up does not occur, $m(t)=m_0e^{-\sqrt{2\diffcst}t}$ and $a(t)=a_0e^{\sqrt{2\diffcst}t}$. In case $(iii)$ blow up does not occur and conclusion follows from~\eqref{eq:sol-m-v} and~\eqref{eq:sol-a-v}.
\end{proof}

\begin{figure}
    \centering
    \includegraphics[width=0.62\textwidth]{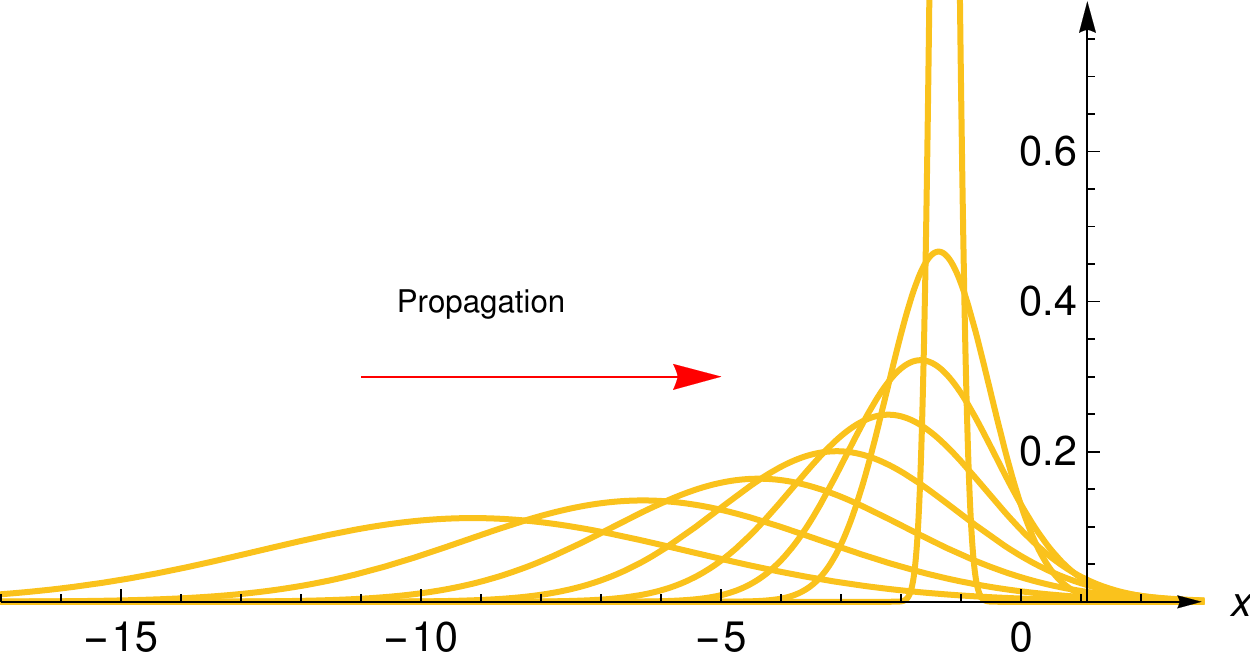}
    \caption{Case $(i)$ concentration in f{}inite time.  The values of the parameters are $a_0=5/64$, $m_0=-585/64$, $\diffcst=1$. It follows that $T^\star\approx 1.878$ and $m(T^\star)\approx -1.277<0$.}
    \label{fig:concentrationFiniteTime}
\end{figure}

\begin{figure}
    \centering
    \includegraphics[width=0.62\textwidth]{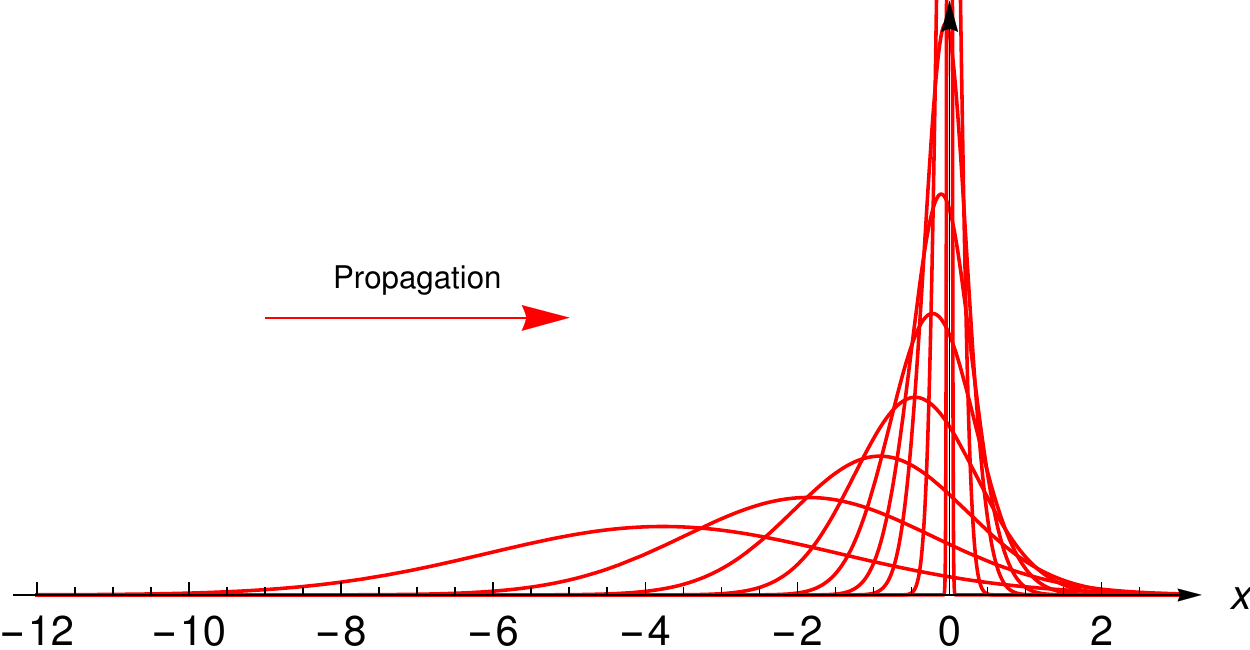}
    \caption{Case $(ii)$ concentration in inf{}inite time: the solution converges to a Dirac mass at zero. The values of the parameters are $a_0=3/16$, $m_0=-8\sqrt{2}/3$, $\diffcst=1$.}
    \label{fig:concentrationInfiniteTime}
\end{figure}

\begin{figure}
    \centering
    \includegraphics[width=0.7\textwidth]{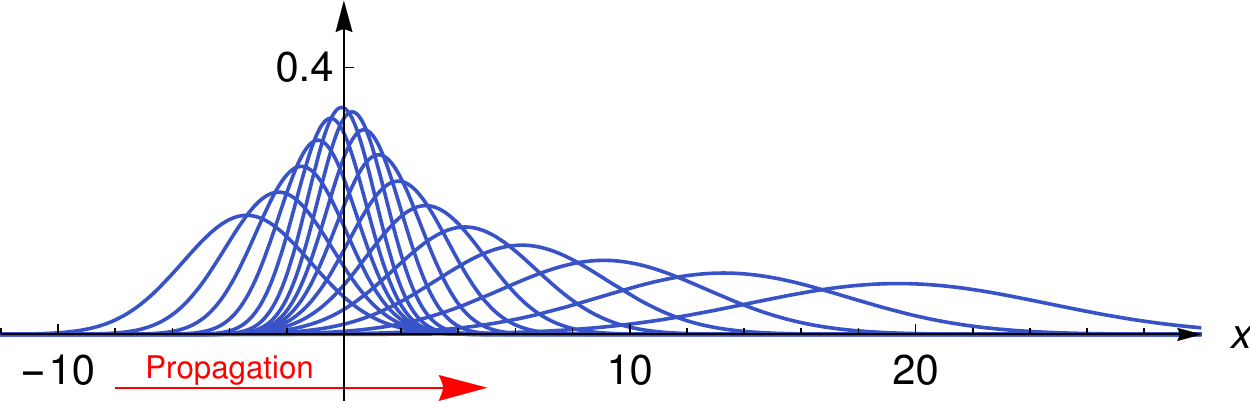}
    \caption{Case $(iii)$ with $m_0<0$, anti-dif{}fusion/dif{}fusion behaviour. The values of the parameters are $a_0=2/10$, $m_0=-34/10$, $\diffcst=1$.}
    \label{fig:focusinDefocusing}
\end{figure}

\begin{figure}
    \centering
    \includegraphics[width=0.7\textwidth]{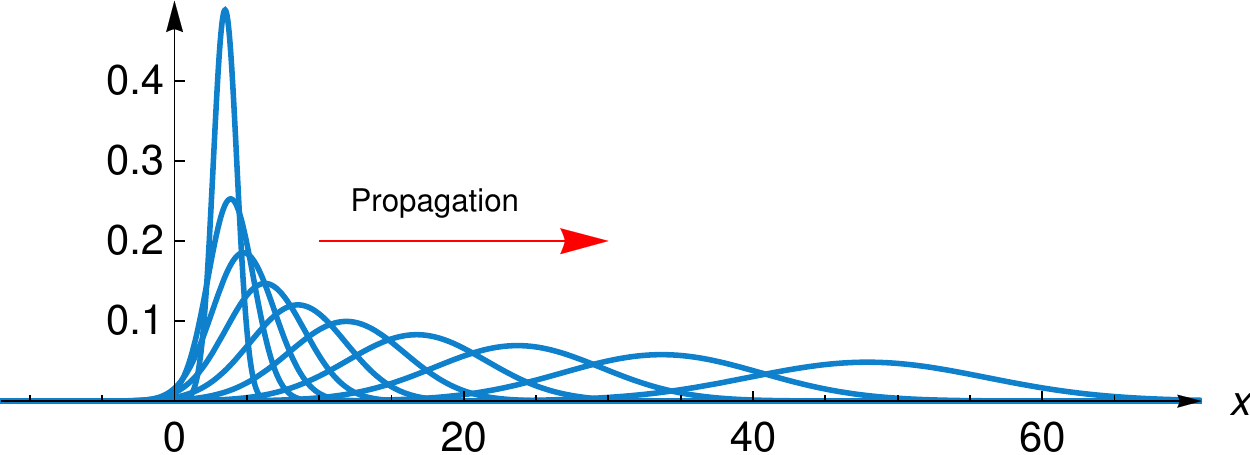}
    \caption{Case $(iii)$ with $m_0\geq 0$, the solution is flattening and accelerating. The values of the parameters are $a_0=3/2$, $m_0=7/2$, $\diffcst=1$.}
    \label{fig:accelerating}
\end{figure}

\subsection{General solutions}\label{ss:general-v}

Here we consider a non-negative initial data $v_0\in L^1(\R)$ satisfying $\int_\R v_0(x)\,dx=1$, \eqref{eq:def-queues-legere} and $m_0:=\int _\R xv_0(x)\,dx>0$, and investigate possible solutions $v=v(t,x)$ of~\eqref{eq:v-equation} starting from $v_0$. From Theorem~\ref{th:main}, we are equipped with the unique solution $u=u(t,x)$ of equation~\eqref{eq:principale} starting from $v_0$. Roughly speaking, the understanding of~\eqref{eq:v-equation} now reduces to that of the \textsc{o.d.e.} Cauchy problem
\begin{equation}\label{eq:ode}
\begin{cases}
\varphi'(t)  =\ubar(\varphi(t)) \text{ for }t>0,\\
\varphi(0)  = 0.
\end{cases} 
\end{equation}
Indeed, def{}ining 
\begin{equation}\label{eq:v-u}
v(t,x):=u(\varphi(t),x),
\end{equation}
we have $\overline v(t)=\overline u (\varphi(t))$ and one checks that, since $u$ solves~\eqref{eq:principale} (and starts from $v_0$), $v=v(t,x)$ solves~\eqref{eq:v-equation} (and starts from $v_0$) as long as the solution $\varphi(t)$ of~\eqref{eq:ode} exists.

\begin{lem}\label{lem:ODE} The solution $\varphi(t)$ of the \textsc{o.d.e.} Cauchy problem~\eqref{eq:ode} is global.
\end{lem}

\begin{proof} If $\varphi$ blows up in f{}inite time, say at time $T>0$, then we obtain from the \textsc{o.d.e.} that \[\int _0^{+\infty}\frac{dy}{\overline u(y)}=T.\]
Hence we deduce from~\eqref{eq:implicit-ubar} that \[\overline u(t)\leq \Vert {C_0}'\Vert _{L^{\infty}(0,T)}+2\diffcst \int _0 ^{t}\frac{y}{\overline u (y)}dy.\] But, from~\eqref{eq:ubar-egal}, we know that $\frac y{\overline u(y)}\leq \frac{1}{\sqrt{2\diffcst}}$ and thus \[
\overline u(t)\leq \Vert {C_0}'\Vert _{L^{\infty}(0,T)}+\sqrt{2\diffcst}\,t,\]
which contradicts $\frac{1}{\overline u}\in L^{1}(0,+\infty)$. This concludes the proof of the lemma and, clearly, of estimate~\eqref{eq:not-int}. 
\end{proof}

Since $\overline u$ is smooth, so is $\varphi$. Also, $\varphi'(t)\geq m_0>0$ and, in particular, $\varphi(t)$ tends to $+\infty$ as $t\to+\infty$. Hence $\varphi:(0,+\infty)\rightarrow(0,+\infty)$ is a smooth dif{}feomorphism. In other words,~\eqref{eq:v-u} shows that there is a one-to-one correspondence between solutions to~\eqref{eq:v-equation} and that of~\eqref{eq:principale}. We omit the full details but this clearly concludes the analysis of the Cauchy problem associated with~\eqref{eq:v-equation}, whose surprising qualitative behaviours have already been underlined in subsection~\ref{ss:gauss-v}.

\bigskip

\noindent{\bf Acknowledgements.} The authors are very grateful to Anton Zadorin for sharing his expertise on the models appearing in \cite{Zad-17}, and for a precise relecture of the preprint, in particular the introduction. M. Alfaro is supported by the 
ANR \textsc{i-site muse}, project \textsc{michel} 170544IA (n$^{\circ}$ ANR \textsc{idex}-0006). M. Veruete is supported by the National Council for Science and Technology of Mexico.

\bibliographystyle{siam}  
\bibliography{biblio}

\end{document}